\documentclass[leqno,11pt]{amsart}

\usepackage[colorlinks,linkcolor=red,citecolor=green,pagebackref,hypertexnames=false]{hyperref}
\usepackage{amsmath,amsthm,amssymb,mathrsfs}
\usepackage[margin=1in]{geometry}
\usepackage[english]{babel}
\usepackage{bookmark}
\usepackage{graphicx}
\usepackage{xcolor}

\numberwithin{equation}{section}

\newcommand{\bbR}{\mathbb{R}}
\DeclareMathOperator\degmap{deg}

\newtheorem{theorem}{Theorem}
\newtheorem{proposition}{Proposition}
\newtheorem{lemma}{Lemma}

\theoremstyle{definition}
\newtheorem{remark}{Remark}

\title{Critical points of point charge potentials along lines}

\author{Gon\c{c}alo Oliveira}
\address{Centro de Análise Matemática, Geometria e Sistemas Dinâmicos, Departamento de Matemática, Instituto Superior Técnico, Universidade de Lisboa, Av. Rovisco Pais 1, 1049-001 Lisbon, Portugal}

\urladdr{\href{https://sites.google.com/view/goncalo-oliveira-math-webpage/home}{sites.google.com/view/goncalo-oliveira-math-webpage/home}}
\email{\href{mailto:galato97@gmail.com}{galato97@gmail.com}}
\date{}


\begin{document}
	\raggedbottom
	
	\begin{abstract}
		We prove Conjecture 1.9 of Gabrielov--Novikov--Shapiro. More precisely, for every $p>0$, every nonconstant restriction to a line of a potential generated by $n$ point charges has at most $2n-1$ critical points. The bound is sharp, even for positive charges. The proof uses a duality with an auxiliary planar potential generated by collinear charges and elementary degree theory.
	\end{abstract}
	
	\maketitle
	
	\section{Introduction}
	
	\subsection{Context}
	
	The problem of estimating the number of equilibrium points created by finitely many point charges goes back at least to Maxwell \cite{Maxwell} and Morse \cite{MorseCairns}. Given points $A_1,\ldots,A_n\in\bbR^3$ and real numbers $\zeta_1,\ldots,\zeta_n$, consider the family of potentials
	\begin{equation}\label{eq:potential-intro}
		V_p(X)=\sum_{j=1}^n \frac{\zeta_j}{|X-A_j|^p},\qquad p>0.
	\end{equation}
	The usual Coulomb potential corresponds to $p=1$. The classical problem is to bound the number of zeroes of $\nabla V_p$ in terms of the number of charges, and even in this case the expected general bound remains unknown. See \cite{GNS,EFO} and the references therein. Very recently, Arathoon--Ball--Kvalheim \cite{ABK} disproved the classical Maxwell bound $(n-1)^2$ by constructing five point charges with at least $24$ nondegenerate equilibria. 
	
	A natural simpler problem is that of restricting the potential to a straight line. Let $L\subset\bbR^3$ be such a line and choose an affine coordinate $x$ on $L$. Let $a_j$ be the coordinate of the orthogonal projection of $A_j$ onto $L$ and $b_j$ that of the orthogonal projection onto $L^\perp$. Then
	\begin{equation}\label{eq:restriction-intro}
		V_p|_L(x)=\sum_{j=1}^n \frac{\zeta_j}{\big((x-a_j)^2+b_j^2\big)^{p/2}},
	\end{equation}
	which is the potential of the one-dimensional problem considered by Gabrielov--Novikov--Shapiro in Conjecture 1.9 of \cite{GNS}. It asserts that, for $p\geq1$, the function in \eqref{eq:restriction-intro} has at most $2n-1$ critical points. This conjecture was recently recalled in \cite{EFO}, but the point of Conjecture 1.9 in \cite{GNS} is that the upper bound $2n-1$ does not depend on $p$.
	
	After the initial submission of this manuscript to arXiv, but before it became public, I learned of independent work of Duan \cite{Duan}, submitted a few days earlier, which also proves Conjecture 1.9. The present version has been revised to acknowledge and compare with that work.
	Duan proves a more general projective paired Haar theorem using Gegenbauer and Wronskian calculations together with methods from the theory of locally convex projective curves. The proof given here takes a comparatively shorter route specifically directed to the point-charge problem. It transforms our one-dimensional problem into a two-dimensional one for a potential generated by collinear charges. A local analysis of its critical points, followed by perturbing to the non-degenerate case, allows us to count the critical points from simple (local and global) degree arguments. The two approaches share some local analysis, but their global mechanisms are different.

	\subsection{Main result}
	
	Our main result, Theorem \ref{thm:main} proves Conjecture 1.9 of \cite{GNS}. It is however, interesting to note that the ambient dimension plays no role after restricting the potential to a line. Hence, the same statement holds for any such potential in $\bbR^d$, with $d\geq2$.
	
	There are two minor points which should be kept in mind from the beginning. First, if two sources have the same projection $a_j$ and the same distance $|b_j|$ from $L$, then they induce exactly the same denominator and their charge strengths can be added. Secondly, if a source lies on $L$, then $b_j=0$ and the restriction is singular at $x=a_j$. By a critical point of the restriction we shall always mean a critical point in its domain.
	
	\begin{theorem}\label{thm:main}
		Let $p>0$ and let $V_p$ be the potential in \eqref{eq:potential-intro}. Fix an affine line $L\subset\bbR^3$ and combine all charges which induce the same pair $(a_j,|b_j|)$ in \eqref{eq:restriction-intro} by adding their coefficients, and omit those for which the resulting coefficient vanishes. Suppose that $N$ nonzero charges remain. If $N\geq1$, then the restriction $V_p|_L$ has at most
		\[
		2N-1\leq 2n-1
		\]
		distinct critical points in its domain. If $N=0$, then $V_p|_L\equiv0$.
	\end{theorem}
	
	As we shall explain below. This main theorem is actually a Corollary of the following result.
	
	\begin{theorem}\label{thm:Haar}
		Let $p>0$, $\{(a_j,|b_j|)\}_{j=1,\ldots,N} \subset \bbR\times[0,+\infty)$ pairwise distinct, and set
		\[
		R_j(x)=(x-a_j)^2+b_j^2.
		\]
		Then, any nonzero linear combination of the $2N$ functions
		\begin{equation}\label{eq:Haar-family}
			R_j(x)^{-(p+2)/2},\qquad
			(x-a_j)R_j(x)^{-(p+2)/2},\qquad j=1,\ldots,N,
		\end{equation}
		has at most $2N-1$ distinct zeroes in $D=\bbR\backslash\{a_j \ | \ b_j=0\}$. In particular, when all $b_j>0$, the functions in \eqref{eq:Haar-family} form a Haar system on $\bbR$.
	\end{theorem}
	
	To see why Theorem \ref{thm:Haar} implies Theorem \ref{thm:main}, start by combining charges with equal profiles. This leaves us with $N\geq1$ (the case $N=0$ is immediate) charges whose strengths we will still denote by $\zeta_j\neq0$. Thus
	\[
	V_p|_L(x)=\sum_{j=1}^N \zeta_jR_j(x)^{-p/2},
	\]
	and
	\begin{equation}\label{eq:restricted-derivative-intro}
		(V_p|_L)'(x)
		=-p\sum_{j=1}^N \zeta_j(x-a_j)R_j(x)^{-(p+2)/2}.
	\end{equation}
	This is a nonzero linear combination of the second function in each pair in \eqref{eq:Haar-family}. Hence Theorem \ref{thm:Haar} gives at most $2N-1$ zeroes of the derivative in $D$, which are precisely the critical points of the restriction. Thus we only need to prove Theorem \ref{thm:Haar}.
	
	The proof is by contradiction. Suppose that a nonzero linear combination of the functions in \eqref{eq:Haar-family} vanishes at $2N$ distinct points $x_1,\ldots,x_{2N}\in D$. Then, their values at these points can be used to construct a singular $2N\times2N$ matrix $C$, whose $j$-th pair of columns is obtained by evaluating the functions in \eqref{eq:Haar-family} at the points $x_1,\ldots,x_{2N}\in D$.
	Because it consists in evaluating the above functions at these points, we shall refer to $C$ as the evaluation matrix.
	
	Because $C$ is singular, there is a nonzero vector $q=(q_1,\ldots,q_{2N})\in\ker C^T$. Omitting its zero entries and regarding the remaining $q_r$ as signed charges at the collinear points $(x_r,0)$ we can consider the planar potential
	\begin{equation}\label{eq:aux-intro}
		U(a,b)=\sum_r \frac{q_r}{\big((a-x_r)^2+b^2\big)^{p/2}}.
	\end{equation}
	The equations $C^Tq=0$ imply that, when $b_j \neq 0$, both $(a_j,\pm b_j)$ are critical points of $U$. On the other hand, when $b_j=0$, they imply that $(a_j,0)$ is a critical point and that its Hessian vanishes. This is also where both functions in each pair in \eqref{eq:Haar-family} enter the proof. After passing to $\ker C^T$, the second columns give $U_a=0$, while the first give $U_b=0$ away from the axis and the additional degeneracy on the axis.
	
	It remains to understand what happens to these prescribed critical points under a small perturbation of the auxiliary charge strengths. The function $U$ satisfies
	\begin{equation}\label{eq:weighted-intro}
		U_{aa}+U_{bb}+\frac{p}{b}U_b=0.
	\end{equation}
	Section \ref{sec:local} shows that every critical point is isolated and that, if the first nonconstant homogeneous term at such a point has degree $k\geq2$, then every sufficiently small perturbation for which the nearby critical points are nondegenerate has exactly $k-1$ such critical points, all saddles. Thus an arbitrary critical point gives at least one saddle, while one with vanishing Hessian gives at least two.
	
	Section \ref{sec:morse} shows that the charge strengths can be perturbed so that all critical points are nondegenerate and the total charge is nonzero. If the number of nonzero charges $q_r$ is $m\leq 2N$, this perturbation can be made still with $m$ nonzero charges and the resulting (non-degenerate) potential has exactly $m-1$ critical points, as can be proven by degree arguments (or Morse theory). Our perturbation argument will show that the $N$ profiles $\{(a_j,b_j)\}_{j=1,\ldots , N}$ give rise to at least $2N$ critical points (two from each pair of profiles with $\pm b_j\neq 0$, and at least two from each profile with $b_j=0$). Hence
	\[
	2N\leq m-1,
	\]
	which is impossible as $m \leq 2N$.

	\begin{remark}[The bound is sharp]\label{rem:sharp} The estimate in Theorem \ref{thm:main} cannot be improved in this generality. This is true, even for equal positive charges and for a line disjoint from them. For $N=1$ this statement is immediate. For $N\geq2$, choose $a_1<\cdots<a_N$ and consider
		\[
		V_\epsilon(x)=\sum_{j=1}^N\frac{1}{\big((x-a_j)^2+\epsilon^2\big)^{p/2}}.
		\]
		This is the restriction to the $x$-axis of the potential generated by equal positive charges at $(a_j,\epsilon,0)$. The contribution of the $j$-th charge to $V_\epsilon'$ at $a_j\pm\epsilon$ is $\mp p\,2^{-(p+2)/2}\epsilon^{-p-1}$, while the remaining terms stay bounded as $\epsilon\to0$. Hence, for $\epsilon$ sufficiently small,
		\[
		V_\epsilon'(a_j-\epsilon)>0,
		\qquad
		V_\epsilon'(a_j+\epsilon)<0.
		\]
		Taking $\epsilon$ smaller than half the minimum distance between the $a_j$, there is one zero of $V_\epsilon'$ in each interval $(a_j-\epsilon,a_j+\epsilon)$ and one in each interval $(a_j+\epsilon,a_{j+1}-\epsilon)$. These give $2N-1$ distinct critical points, and Theorem \ref{thm:main} shows that there can be no others.
	\end{remark}
	
	\begin{remark}[Purely one-dimensional proof?] It would be interesting to obtain this bound from a purely one-dimensional proof working with $V_p|_L$ rather than through the duality with the two-dimensional problem.
	\end{remark}
	
	\subsection*{Acknowledgments}
	
	I thank Boris Shapiro and Dmitry Novikov for useful comments. I am particularly grateful to Boris Shapiro for drawing my attention to the recent independent work of Xiuqing Duan, and to Dmitry Novikov for his comments that led to the discussion preceding Proposition \ref{prop:morse-count}. I also thank Chris Fillmore and Herbert Edelsbrunner for countless conversations on the topic. The author acknowledges funding from FCT projects 10.54499/UID/04459/2025, 10.54499/UID/PRR/04459/2025.
	
	\subsection*{Tool and computational resource disclosure}
	
	Artificial intelligence was used extensively in the development of this article. In the spirit of the Leiden Declaration on Artificial Intelligence and Mathematics \cite{LeidenDeclaration}, I record its role here.
	
	The project arose during an extended interaction with ChatGPT which initially concerned a different problem. Its use was mathematically substantive, and several parts of the strategy and proof were developed through interaction with the system. The arguments were subsequently modified, simplified, and checked by me. ChatGPT was also used to help reorganize and edit the exposition and the \LaTeX\ source.
	
	I take full responsibility for the correctness of the results and proofs.

	\section{Local analysis of the critical points}\label{sec:local}
	
	Fix $p>0$, distinct real numbers $x_1,\ldots,x_m$, and nonzero real coefficients $q_1,\ldots,q_m$. In this and the next section we will be considering
	\begin{equation}\label{eq:U}
		U(a,b)=\sum_{r=1}^m \frac{q_r}{\big((a-x_r)^2+b^2\big)^{p/2}}
	\end{equation}
	on
	\[
	\Omega=\bbR^2\backslash\{(x_1,0),\ldots,(x_m,0)\}.
	\]
	The function $U$ is real analytic on $\Omega$ and even in $b$. For each term in the sum, a direct computation 
	\(
	\left(\partial_a^2+\partial_b^2+\frac{p}{b}\partial_b\right)
	\big((a-x)^2+b^2\big)^{-p/2}=0,
	\)
	for $b \neq 0$. Linearity then gives 
	\begin{equation}\label{eq:weighted}
		U_{aa}+U_{bb}+\frac{p}{b}U_b=0
	\end{equation}
	away from the $a$-axis, i.e. $\{b =0\}$. However, as $U$ is even in $b$, $U_b$ is odd and thus $U_b/b$ extends real analytically to all of $\Omega$, and equation \eqref{eq:weighted} extends as well, with the last term interpreted by this extension.\footnote{
		Equation \eqref{eq:weighted} is the generalized axially symmetric potential equation of Weinstein \cite{Weinstein}. When $p$ is a nonnegative integer, it is the ordinary Laplace equation on $\bbR^{p+2}$ restricted to functions invariant under rotations in the last $p+1$ variables.
	}
	
	At an off-axis critical point the first nonconstant homogeneous term of the Taylor expansion will be an ordinary harmonic polynomial. On the axis it satisfies the corresponding generalized axially symmetric equation. Our first Lemma below (Lemma \ref{lem:homogeneous}) gives the local model at points on the axis and computes the degree of its gradient. Recall that, if $z_0$ is an isolated critical point of a function $f$ and $B$ is a sufficiently small disk centered at $z_0$, its local degree at $z_0$ is defined by
	\[
	\degmap\left(\frac{\nabla f}{|\nabla f|}:\partial B\longrightarrow S^1\right).
	\]
	This quantity is usually called the index of the vector field $\nabla f$ at $z_0$, but we will avoid this terminology in order avoid confusion with the Morse index of $f$. In fact, if $z_0$ is nondegenerate, its Morse index $\lambda$ is related to this degree through
	\[
	\degmap\left(\frac{\nabla f}{|\nabla f|} \right)= \operatorname{sign}\det D^2f(z_0)=(-1)^\lambda.
	\]
	In particular, a saddle has local degree $-1$. See, for example, \cite[\S\S5--6]{Milnor}. This latest observation will be important in the proof of Proposition \ref{prop:local-perturbation} when we combine it with our local analysis to study what happens when $U$ is perturbed to a Morse function.
	
	\begin{lemma}\label{lem:homogeneous}
		Let $k\geq1$ and $p\geq0$. The space of homogeneous polynomials $P(a,b)$ of degree $k$, even in $b$, and satisfying \eqref{eq:weighted}, i.e. $P_{aa}+P_{bb}+\frac{p}{b}P_b=0$, is one-dimensional. If $P_{k,p}$ is normalized so that the coefficient of $a^k$ is $1$, then $\nabla P_{k,p}$ does not vanish on $\bbR^2\backslash\{0\}$ and, on the positively oriented unit circle,
		\begin{equation}\label{eq:degree-homogeneous}
			\degmap\left(\frac{\nabla P_{k,p}}{|\nabla P_{k,p}|}\right)=1-k.
		\end{equation}
	\end{lemma}
	
	\begin{proof}
		Every homogeneous polynomial of degree $k$ which is even in $b$ can be written as
		\begin{equation}\label{eq:Pexpansion}
			P(a,b)=\sum_{j=0}^{\lfloor k/2\rfloor}c_j a^{k-2j}b^{2j}.
		\end{equation}
		The coefficient of $a^{k-2j-2}b^{2j}$ in $P_{aa}+P_{bb}+\frac{p}{b}P_b$ is
		\[
		(k-2j)(k-2j-1)c_j+(2j+2)(2j+1+p)c_{j+1}.
		\]
		Since $2j+1+p>0$ for $p\geq0$, setting these coefficients equal to zero determines $c_{j+1}$ from $c_j$ for $0\leq j<\lfloor k/2\rfloor$. Thus $c_0$ determines all the coefficients and the space is one-dimensional. Furthermore, observe that the coefficients depend continuously on $p$.
		
		At this point, it is convenient to write $P_{k,p}(r\cos\theta,r\sin\theta)=r^kF(\theta)$. Then, equation \eqref{eq:weighted} for $P_{k,p}$ turns into a second-order linear ODE for $F$, which is nonsingular for $b\neq0$. If $\nabla P_{k,p}$ vanished somewhere away from the $a$-axis, then the radial and angular derivatives would both vanish, resulting in $F=F'=0$. Uniqueness for the ODE implies that $F$ would vanish on an interval, and $P_{k,p}$ would be zero on an open set, a contradiction. On the $a$-axis, the normalization gives $P_{k,p}(a,0)=a^k$, and hence $(P_{k,p})_a(a,0)=ka^{k-1}$, which non-zero for $a\neq0$.
		
		As the restriction of $\nabla P_{k,p}$ to the unit circle varies continuously with $p\geq0$ through nonvanishing maps, its degree is constant in $p$ and we may compute it at $p=0$. In that case, the equation is $\Delta P=0$ and the polynomial $\operatorname{Re}(a+ib)^k$ is homogeneous of degree $k$, harmonic, even in $b$, and has coefficient $1$ in front of $a^k$. By uniqueness, we must have
		\[
		P_{k,0}(a,b)=\operatorname{Re}(a+ib)^k.
		\]
		Furthermore, identifying $\mathbb R^2 \cong \mathbb C$ via $(v_1,v_2) \mapsto v_1+iv_2$ and writing $a+ib=e^{i\theta}$, we have
		\[
		(P_{k,0})_a+i(P_{k,0})_b=ke^{-i(k-1)\theta},
		\]
		whose degree is $1-k$, proving \eqref{eq:degree-homogeneous}.
	\end{proof}
	
	We can now describe how an arbitrary critical point behaves under a small perturbation of the charge strengths.
	
	\begin{proposition}\label{prop:local-perturbation}
		Let $z_0\in\Omega$ be a critical point of $U$ and $k$ the order of vanishing of $U-U(z_0)$. 
		Then $k\geq2$, $z_0$ is isolated, and there is a closed disk $B\subset\Omega$ centered at $z_0$ such that every sufficiently small perturbation of the charge strengths resulting in a nondegenerate $\widetilde U$, has exactly $k-1$ critical points in $B$, all saddles. In particular, every such perturbation has at least one critical point in $B$, and at least two if $D^2U(z_0)=0$.
	\end{proposition}
	
	\begin{proof}
		Since $U$ is real analytic in $\Omega$ and nonconstant, its Taylor expansion at $z_0$ has a first nonconstant homogeneous term $P_k$, with $k\geq2$:
		\begin{equation}\label{eq:Taylor}
			U(z)=U(z_0)+P_k(z-z_0)+O(|z-z_0|^{k+1}).
		\end{equation}
		
		If $z_0$ is off the axis, the terms of degree $k-2$ in \eqref{eq:weighted} give $\Delta P_k=0$. Thus, $P_k$ is a linear combination of $\operatorname{Re}((a+ib)^k)$ and $\operatorname{Im}((a+ib)^k)$, so its gradient is nonzero away from the origin and has degree $1-k$. If $z_0$ lies on the axis, then $P_k$ is even in $b$ and satisfies $(P_k)_{aa}+(P_k)_{bb}+\frac{p}{b}(P_k)_b=0$, so we can use
		Lemma \ref{lem:homogeneous} to get to the same conclusion.
		
		Differentiating \eqref{eq:Taylor} gives
		\[
		\nabla U(z_0+y)=\nabla P_k(y)+O(|y|^k).
		\]
		Since $\nabla P_k$ is non-vanishing and homogeneous of degree $k-1$, $|\nabla U(z_0+y)|\gtrsim |y|^{k-1}$ for sufficiently small $|y|$ and so it has no zero in a sufficiently small ball around $z_0$. In particular, we can homotope $\nabla P_k$ to $\nabla U$ through never vanishing vector fields in a small circle around $z_0$. Hence, $z_0$ is the only zero in a small ball $B$ around $z_0$ and, using again Lemma \ref{lem:homogeneous},
		\[
		\degmap\left(\frac{\nabla U}{|\nabla U|}: \partial B \to S^1 \right)=1-k.
		\]
		By continuous dependence on the charge strengths, for $(\tilde q_1,\ldots, \tilde q_m)$ sufficiently close to $(q_1,\ldots,q_m)$ the gradient of the perturbed potential $\widetilde U$ remains never vanishing on $\partial B$, and the local degree is preserved.
		
		Furthermore, assuming all critical points of $\widetilde U$ in $B$ to be nondegenerate. Every such critical point is a saddle. Indeed, off the axis \eqref{eq:weighted} makes the Hessian traceless, while on the axis evenness and the fact that $\widetilde U_b/b$ extends there with value $\widetilde U_{bb}$ give
		\[
		\widetilde U_{ab}=0,
		\qquad
		\widetilde U_{aa}+(p+1)\widetilde U_{bb}=0.
		\]
		Thus every critical point of $\widetilde U$ in $B$ has local degree $-1$. Homotopy invariant and additivity of the degree gives
		\[
		1-k=-\#\{\text{critical points of }\widetilde U\text{ in }B\},
		\]
		which proves the asserted count. Finally, if $D^2U(z_0)=0$, the quadratic term in the Taylor expansion vanishes, so $k\geq3$.
	\end{proof}
	
	\section{Perturbation to the Morse case}\label{sec:morse}
	
	In this section we first show that the charge strengths can always be perturbed so that all critical points are nondegenerate and the total charge is nonzero. Then, in Proposition \ref{prop:morse-count} we use degree theory to count the critical points of such a potential and Proposition \ref{prop:local-perturbation}. 
	
	To make the dependence on $q=(q_1,\ldots,q_m)\in\bbR^m$ explicit, in this section we shall write
	\begin{equation}\label{eq:Uq}
		U_q(a,b)=\sum_{r=1}^m \frac{q_r}{\big((a-x_r)^2+b^2\big)^{p/2}}.
	\end{equation}
	
	\begin{lemma}\label{lem:generic-morse}
		Let $p>0$, $x_1,\ldots,x_m$ all distinct, and $q=(q_1,\ldots,q_m)$ with every $q_r\neq0$. Then, in any neighborhood $\mathcal O \subset \bbR^m$ of $q$, there is $\widetilde q\in\mathcal O$ such that every $\widetilde q_r \neq 0$, $\sum_r\widetilde q_r\neq0$, and all critical points of $U_{\widetilde q}$ are nondegenerate.
	\end{lemma}
	
	\begin{proof}
		When $m=1$ the statement is immediate, so suppose that $m\geq2$. We first show that
		\[
		\mathcal Z=\{(q,a,b)\in\bbR^m\times\bbR \times (\bbR\backslash \{0\}):\nabla U_q(a,b)=0\}
		\]
		is a smooth $m$-dimensional manifold away from $\{b=0\}$. In order to show that using the implicit function theorem we need only show that $D_q(\nabla U_q)$ has rank two whenever $b \neq 0$. To do that notice that
		\[
		\frac{\partial}{\partial q_r } \left(\nabla U_q\right)=p\frac{(x_r-a,-b)}{\big((a-x_r)^2+b^2\big)^{(p+2)/2}},
		\]
		and so, for $r \neq s$, the determinant of $\partial_{q_r}(\nabla U_q)$ and $\partial_{q_s}(\nabla U_q)$ is
		\[
		p^2b\,
		\frac{x_s-x_r}
		{\big((a-x_r)^2+b^2\big)^{(p+2)/2}
			\big((a-x_s)^2+b^2\big)^{(p+2)/2}} \neq 0,
		\]
		for $b \neq 0$, as claimed.
		
		Let $\pi:\mathcal Z\to\bbR^m$ be projection onto the charge strengths. At a point $(q,a,b)\in\mathcal Z$, with $b \neq 0$, a tangent vector $(v,w)\in\bbR^m\times\bbR^2$ satisfies
		\[
		D_q(\nabla U_q)v+D^2U_q\,w=0.
		\]
		Since $D_q(\nabla U_q)$ is onto $\bbR^2$, the differential of $\pi$, $D\pi(v,w)=v$, is onto if and only if $D^2U_q(a,b)$ is invertible. Sard's theorem \cite[Chapter~3]{Hirsch} therefore implies that, outside a set of measure zero in $\bbR^m$, every critical point of $U_q$ with $b \neq 0$ is nondegenerate. 
		
		Since $U_q$ is even in $b$, its critical points on the axis are exactly the zeroes of $a\mapsto(U_q)_a(a,0)$. For the axis, let
		\[
		\mathcal Z_0=\{(q,a)\in\bbR^m\times(\bbR\backslash\{x_1,\ldots,x_m\}):(U_q)_a(a,0)=0\}.
		\]
		For every $r$,
		\[
		\frac{\partial}{\partial q_r}(U_q)_a(a,0)
		=p(x_r-a)|x_r-a|^{-p-2}\neq0.
		\]
		and, again by the implicit function theorem, $\mathcal Z_0$ is also a smooth $m$-dimensional manifold. If $\pi_0:\mathcal Z_0\to\bbR^m$ denotes projection onto the charge strengths, the same tangent-space argument shows that $q$ is a regular value of $\pi_0$ if and only if $(U_q)_{aa}(a,0)\neq0$ at every axis critical point. Sard's theorem again shows that this holds outside a set of measure zero.
		
		At such an axis critical point, evenness gives $(U_q)_{ab}=0$. Since $(U_q)_b/b$ extends across the axis with value $(U_q)_{bb}$, equation \eqref{eq:weighted} gives
		\begin{equation}\label{eq:axis-Hessian-relation}
			(U_q)_{aa}+(p+1)(U_q)_{bb}=0.
		\end{equation}
		Since $(U_q)_{aa}\neq0$, the Hessian is therefore nonsingular.
		
		The union of the two exceptional sets given by Sard's theorem and the hyperplanes $q_r=0$ and $\sum_rq_r=0$ has measure zero in $\bbR^m$. Every neighborhood $\mathcal O$ therefore contains a vector $\widetilde q$ outside this union, and such a vector satisfies the conclusion.
	\end{proof}
	
	The following count is the two-dimensional version of the argument used by Maxwell in his discussion in \cite[Chapter VI]{Maxwell}. Maxwell's argument is, in spirit, an early form of Morse theory showing that the absence of local maxima and minima constrains the number of critical points via the topology of the level sets. See the discussion in the Appendix of \cite{GNS} for more details.
	
	\begin{proposition}\label{prop:morse-count}
		Let $p>0$, locations $x_1 , \ldots , x_m \in \mathbb{R}$ be pairwise distinct, and $q_1, \ldots , q_m$ all nonzero, and satisfying $Q:=\sum_{r=1}^m q_r\neq0$. Further suppose that all critical points of the potential 
		\[
		U(a,b)=\sum_{r=1}^m \frac{ q_r}{\big((a-x_r)^2+b^2\big)^{p/2}},
		\]
		are nondegenerate. Then, $U$ has exactly $m-1$ critical points, all of which are saddles.
	\end{proposition}
	
	\begin{proof}
		At an off-axis critical point, \eqref{eq:weighted} gives $U_{aa}+U_{bb}=0$, while an axis critical point, evenness and \eqref{eq:axis-Hessian-relation} give $U_{ab}=0$ and $U_{aa}+(p+1)U_{bb}=0$. Since the Hessian is nonsingular by assumption, in both cases it is indefinite and all critical points are saddles.
		
		The rest of the proof is a variation of the classical Morse--Cairns argument for electrostatic potentials in $\bbR^3$, see \cite[Chapter~32]{MorseCairns}. As $Q\neq0$, for sufficiently large $\rho=\sqrt{a^2+b^2} \gg 1$
		\begin{equation}\label{eq:monopole-asymptotics}
			U=Q\rho^{-p}+O(\rho^{-p-1}),
			\qquad
			\nabla U=-pQ\rho^{-p-1}\partial_\rho+O(\rho^{-p-2}),
		\end{equation}
		while near the $r$-th charge, i.e. for $\rho_r=\sqrt{(a-x_r)^2+b^2} \ll 1$,
		\begin{equation}\label{eq:charge-local-morse}
			U=q_r\rho_r^{-p}+O(1),
			\qquad
			\nabla U=-pq_r\rho_r^{-p-1}\partial_{\rho_r}+O(1).
		\end{equation}
		In particular, it follows that there are neighborhoods of infinity, $\mathbb R^2 \backslash B_R(0)$ and of the charges $B_\varepsilon(x_r,0)$, which contain no critical points. Thus all critical points lie in the compact set
		\[
		K=\overline{B_R(0)}\backslash\bigcup_{r=1}^m B_\varepsilon(x_r,0)
		\]
		contains all the critical points of $U$ and $\nabla U$ does not vanish on $\partial K$. The asymptotics for $\nabla U$ in \eqref{eq:monopole-asymptotics} give that, depending on the sign of $Q$, the normalized gradient on the outer boundary is homotopic to $\pm \partial_\rho$, both of which have degree $1$ with respect to the induced boundary orientation.
		
		Similarly, it follows from \eqref{eq:charge-local-morse} that on a sufficiently small circle around the $r$-th charge the normalized gradient also has degree $1$ on that circle with its counterclockwise orientation. However, as a boundary component of $K$, it carries the opposite orientation, and so has degree $-1$. From additivity of the degree we have
		\[
		\degmap\left(
		\frac{\nabla U}{|\nabla U|}:\partial K\longrightarrow S^1
		\right)=1-m,
		\]
		as the inner boundary has $m$ connected components.
		
		By homotopy invariance and additivity of the degree, this must be the sum of the local degrees of the critical points of $U$ in $K$. Each has local degree $-1$, thus $\# \{ \text{critical points}\} =m-1$.

%
	\end{proof}
	
	\section{The evaluation matrix}\label{sec:evaluation}
	
	We now return to the evaluation matrix introduced in the proof sketch. For pairwise distinct pairs $(a_j,b_j)\in\bbR\times[0,+\infty)$ set
	\begin{equation}\label{eq:Rj-evaluation}
		R_j(x)=(x-a_j)^2+b_j^2.
	\end{equation}
	When $b_j=0$, the corresponding functions are understood only away from $x=a_j$.
	
	\begin{proposition}\label{prop:evaluation}
		Let $p>0$, $\lbrace (a_j,b_j)\rbrace_{j=1,\ldots,N} \subset  \bbR \times [0,+\infty) $ be pairwise distinct, and $x_1,\ldots,x_{2N}\in\bbR$ also distinct. Assume that $x_r\neq a_j$ whenever $b_j=0$. Then
		\begin{equation}\label{eq:evaluation-det}
			\det\left[
			R_j(x_r)^{-(p+2)/2}\ \ \ (x_r-a_j)R_j(x_r)^{-(p+2)/2}
			\right]_{\substack{r=1,\ldots,2N\\ j=1,\ldots,N}}
			\neq0,
		\end{equation}
		where the columns are ordered in the indicated consecutive pairs.
	\end{proposition}
	
	\begin{proof}
		Suppose otherwise, and let $C$ denote the matrix in \eqref{eq:evaluation-det}. Since $C$ is singular, choose a nonzero vector $q=(q_1,\ldots,q_{2N})\in\ker C^T$. Delete the zero entries of $q$, and write $I=\{r:q_r\neq0\}$, and $m=|I|\leq2N$. Then, the collinear potential is simply
		\begin{equation}\label{eq:Uaux}
			U(a,b)=\sum_{r\in I}\frac{q_r}{\big((a-x_r)^2+b^2\big)^{p/2}}.
		\end{equation}

		For each $j$, the two equations which express that $q$ annihilates the $j$-th pair of columns are
		\begin{align}
			\sum_{r\in I}q_rR_j(x_r)^{-(p+2)/2}&=0,
			\label{eq:null-first}\\
			\sum_{r\in I}q_r(x_r-a_j)R_j(x_r)^{-(p+2)/2}&=0.
			\label{eq:null-second}
		\end{align}
		On the other hand,
		\begin{align}
			U_a(a,b)
			&=p\sum_{r\in I}q_r(x_r-a)
			\big((a-x_r)^2+b^2\big)^{-(p+2)/2},
			\label{eq:Ua}\\
			U_b(a,b)
			&=-pb\sum_{r\in I}q_r
			\big((a-x_r)^2+b^2\big)^{-(p+2)/2}.
			\label{eq:Ub}
		\end{align}
		
		If $b_j>0$, equations \eqref{eq:null-first}--\eqref{eq:null-second} and \eqref{eq:Ua}--\eqref{eq:Ub} give $\nabla U(a_j,b_j)=0$ and by evenness, $(a_j,-b_j)$ is also a critical point. If instead $b_j=0$, as none of the $x_r$ is equal to $a_j$, so $(a_j,0)$ is not a singularity of $U$. Equation \eqref{eq:null-second} gives $U_a(a_j,0)=0$, while $U_b(a_j,0)=0$ by evenness in $b$. Thus $(a_j,0)$ is a critical point. In this case equation \eqref{eq:null-first} forces more. Directly differentiating at $b=0$ gives $U_{ab}(a_j,0)=0$ and
		\begin{align*}
			U_{bb}(a_j,0)
			&=-p\sum_{r\in I}q_r|x_r-a_j|^{-p-2},\\
			U_{aa}(a_j,0)
			&=p(p+1)\sum_{r\in I}q_r|x_r-a_j|^{-p-2}.
		\end{align*}
		The common sum in the rights hand sides vanishes by \eqref{eq:null-first} and consequently $
		D^2U(a_j,0)=0$, i.e. it is a degenerate critical point.
		
		Choose $m$ pairwise disjoint balls around each of the points $(a_j,\pm b_j)$ (when $|b_j| \neq 0$) and $(a_j,0)$ (when $b_j=0$). If sufficiently small Proposition \ref{prop:local-perturbation} applies in all those balls, and, by Lemma \ref{lem:generic-morse}, we may perturb the $m$ nonzero charge strengths so that they remain nonzero, their total charge is nonzero, and every critical point of the perturbed potential $\widetilde U$ is nondegenerate. 
		
		Each profile with $|b_j| \neq 0$ gives two balls, each around $(a_j,\pm b_j)$, and containing at least one critical point of $\widetilde U$. Each profile with $b_j=0$ gives one ball centered at the degenerate critical point $(a_j,0)$, where $D^2U(a_j,0)=0$, so Proposition \ref{prop:local-perturbation} gives at least two critical points of $\widetilde U$ in that ball. Hence $\widetilde U$ has at least $2N$ critical points. On the other hand, Proposition \ref{prop:morse-count} gives exactly $m-1$ such critical points. Since $m\leq2N$,
		\[
		2N\leq m-1\leq2N-1,
		\]
		a contradiction. This proves \eqref{eq:evaluation-det}.
	\end{proof}
	
	\begin{proof}[Proof of Theorem \ref{thm:Haar}]
		Suppose that a nonzero linear combination of the functions in \eqref{eq:Haar-family} vanished at $2N$ distinct points $x_1,\ldots,x_{2N}\in D$. If $C$ denotes the corresponding evaluation matrix in Proposition \ref{prop:evaluation}, the coefficients of this linear combination give a nonzero vector in $\ker C$, contradicting the nonsingularity of $C$. Hence there can be at most $2N-1$ distinct zeroes in $D$.
	\end{proof}


\begin{thebibliography}{99}
		
		\bibitem{ABK}
		Philip Arathoon, Gavin Ball, and Matthew D. Kvalheim,
		\textit{The Maxwell Conjecture is False},
		arXiv:2607.27197 (2026).
		
		\bibitem{Duan}
		Xiuqing Duan,
		\textit{Critical Points of Line Restrictions of Signed Point-Charge Potentials},
		arXiv:2608.28376 (2026).
		
		\bibitem{EFO}
		Herbert Edelsbrunner, Christopher D. Fillmore, and Gon\c{c}alo Oliveira,
		\textit{Counting equilibria of the electrostatic potential},
		Proc. London Math. Soc. \textbf{132} (2026), no.~5, e70163.
		
		\bibitem{GNS}
		Andrei Gabrielov, Dmitry Novikov, and Boris Shapiro,
		\textit{Mystery of point charges},
		Proc. London Math. Soc. (3) \textbf{95} (2007), no.~2, 443--472.
		
		\bibitem{Hirsch}
		Morris W. Hirsch,
		\textit{Differential Topology},
		Graduate Texts in Mathematics, Vol.~33,
		Springer-Verlag, New York, 1976.
		
		
		\bibitem{LeidenDeclaration}
		\textit{Leiden Declaration on Artificial Intelligence and Mathematics},
		Zenodo (2026), doi:10.5281/zenodo.20302944.
		
		\bibitem{Maxwell}
		James Clerk Maxwell,
		\textit{A Treatise on Electricity and Magnetism}, Vol.~I,
		republication of the third revised edition, Dover Publications, New York, 1954.
		
		\bibitem{MorseCairns}
		Marston Morse and Stewart S. Cairns,
		\textit{Critical Point Theory in Global Analysis and Differential Topology},
		Pure and Applied Mathematics, Vol.~33, Academic Press, New York--London, 1969.
		
		\bibitem{Milnor}
		John W. Milnor,
		\textit{Topology from the Differentiable Viewpoint},
		Princeton Landmarks in Mathematics, Princeton University Press, Princeton, NJ, 1997.
		
		\bibitem{Weinstein}
		Alexander Weinstein,
		\textit{Generalized axially symmetric potential theory},
		Bull. Amer. Math. Soc. \textbf{59} (1953), 20--38.
		
	\end{thebibliography}
\end{document}